\documentclass[11pt]{article}
\usepackage{amsfonts,amsthm,amsmath,amssymb}
\usepackage{amscd,graphicx,psfrag,subfigure}

\newcommand{\bn}{{\mathbf N}}
\newcommand{\bc}{{\mathbb C}}

\newcommand{\bz}{{\mathbb Z}}

\newcommand{\mi}{{\mathcal{I}}}
\newcommand{\mc}{{\mathcal{C}}}
\newcommand{\mj}{{\mathcal{J}}}
\newcommand{\la}{{\lambda}}

\newtheorem{Theorem}{Theorem} [section]

\newtheorem{Lemma}[Theorem]{Lemma}
\newtheorem{Proposition}[Theorem]{Proposition}
\newtheorem{Remark}[Theorem]{Remark}

\newtheorem*{thmA}{Theorem A}

\newsavebox{\savepar}

\title{On the connectivity of the escaping set for complex exponential Misiurewicz parameters}

\author{{\small Xavier Jarque \thanks{This work has been done while I am on leave from the Departament de Matem\`atica Aplicada i An\`alisi at Universitat de Barcelona. The author is partially supported by grants 2009SGR--792,  MTM2006--05849 and
MTM--2008--01486 Consolider (including a FEDER
contribution)}}\\
{\small Departament d'Enginyeria Inform\`atica i Matem\`atiques} \\ 
{\small Universitat Rovira i Virgili} \\
{\small Avinguda Pa\"isos Catalans, 26} \\
{\small 43007 Tarragona, Catalunya, Spain} \\
{\small E-mail: xavier.jarque@ub.edu}}

\date{\today}

\begin{document}
\maketitle

\abstract{Let $E_{\la}(z)=\la {\rm exp}(z), \ \lambda\in \mathbb C$ be the complex exponential family. For all functions in the family there is a unique asymptotic value at $0$ (and no critical values). For a fixed $\la$, the set of points in $\mathbb C$ with orbit tending to infinity is called the escaping set. 
We prove that the escaping set of $E_{\la}$ with $\la$ Misiurewicz (that is, a parameter for which the orbit of the singular value is strictly preperiodic) is a connected set. }

\section{Introduction}\label{section:intro}

For an entire function $f:\bc\to\bc$ there are two types of points into which some  branches of $f^{-1}$ cannot be continued analytically; these are critical values and  asymptotic values. A {\it critical value} is the image of a critical point (a zero of the derivative of $f$), and $z_{0}\in\bc$ is an {\it asymptotic value} if there is a curve $\alpha(t)$ satisfying  $|\alpha(t)| \to \infty$ and $f\left(\alpha(t)\right)\to z_{0}$, as $t\to \infty$. The closure of the union of critical and asymptotic values is called the set of {\it singular values}.  Singular values are known to play an important role in determining the global dynamics  associated with the iterates of the map; see \cite{CG} (Chapter III) or \cite{Be} (Theorem 7).

Let $(E_{\la})_{\la\in\bc}$ be the complex exponential family, i.e., $E_{\la}(z)=\la  {\rm exp}(z),\ \la \in\bc$. Each complex exponential map is a transcendental entire map with a unique asymptotic value at $z=0$ and no critical values. So, the structure and topology of the Fatou and Julia sets in the dynamical plane strongly depend on the asymptotic behavior of the iterates of the  unique singular value at $z=0$. For this reason $E_{\la}$ is considered the transcendental entire  version of the well--known quadratic polynomial family, $Q_{c}(z)=z^2+c,\ c\in\mathbb C$,  which has, for each $c$,  a unique critical value at $z=c$ (and, of course, no asymptotic values) which determines the structure and topology of the Fatou and Julia sets in the dynamical plane.

However, in contrast to the polynomial case, where all points that tend  to infinity under iteration belong to the basin of attraction of $z=\infty$, and so belong to the Fatou set, the existence of an essential singularity at infinity and a unique finite singular value implies that all points that tend  to infinity under iteration, known as the {\it escaping set} 
$$
\mi\left(E_{\lambda}\right)=\{z\in\bc \ | \  E_{\lambda}^n(z) \to \infty  \},
$$
belong to the Julia set. Precisely it turns out that  $\mj\left(E_{\lambda}\right)=\overline{\mi\left(E_{\lambda}\right)}$; see \cite{EL}.  Notice that for any exponential map the vertical line Re$(z)=\rho$ is sent to the circle of radius $|\la|e^\rho$, so any point $z\in\bc$ whose orbit escapes to infinity  satisfies Re$\left(E_{\la}^n(z)\right) \to \infty$. The direction given by the positive real axis (no restriction on the imaginary part)  is called the {\it asymptotic (escaping) direction}.

As we said, the topology of the Julia and Fatou sets of $E_{\la}$ depend on $\lambda$, as does the dynamical behavior of its orbit. So, for instance, when $\lambda\in(0,1/e)$ the Fatou set is given by a unique, totally invariant immediate basin of attraction, and the Julia set is  a pairwise disjoint union of infinitely  many curves  homeomorphic to $[0,\infty)$ that extend to infinity in the asymptotic direction. 
Each curve has a  unique distinguished point, called {\it endpoint} or {\it landing point}, that does not necessarily belong to the escaping set. For instance for certain curves their endpoints are periodic or preperiodic points, and so have a bounded orbit. See \cite{Ba,B,R1,SZ}. 

In contrast, if the orbit of $0$ escapes (for instance $\lambda>1/e$), or if the orbit of $0$ is  strictly preperiodic  (these $\la$ parameters are usually called {\it Misiurewicz} parameters) the Julia set is the whole plane. 
In this case, the escaping set  also contains (although is not) the curves described  above (see \cite{DJ,DJM,R2,SZ} and references therein for  precise results). 

In 1986  Eremenko  conjectured \cite{E}  that,  for any entire map $f$, all connected components of $\mi(f)$  are unbounded (he had already proven that all connected components of $\overline{\mi(f)}$ are unbounded). 
This conjecture has  recently been proven for a certain class of entire maps, including the exponential family,  but it is false in general (see  \cite{R3S} and \cite{SZ}). Moreover, from the conjecture itself and all later (positive and negative) results, there has been a natural interest in  studying  the topology of the connected components of the escaping set, inclusively for maps of the exponential family.

It is easy to argue that for hyperbolic parameters of the exponential family, like $\lambda\in(0,1/e)$, each of the  above--mentioned curves is a distinct connected component of $\mi(E_{\lambda})$. Recently L. Rempe \cite{R} used the construction of Devaney's indecomposable continua \cite{D} to show that for $\lambda > 1/e$ the escaping set is a connected set of the plane. So,  the infinitely many pairwise disjoint curves that extend to infinity  are just a subset of a unique connected component. 

These results seemed to show that having  a singular value(s) in the Julia set made a crucial difference in the connectedness of the escaping set. However,  H. Mihaljevi\'{c}-Brandt \cite{M-B} proved that, for instance, for any Misiurewicz parameter (a parameter for which the two critical values are preperiodic) in the sine family, $S_{\lambda}(z)=\lambda \sin z$, the escaping set is not a connected subset of the plane. She obtained this result as a corollary of  a much more general theorem on the  conjugacy of quasiconformally equivalent maps (see also \cite{R3}). We note that the sine family has no finite asymptotic values but two critical values given by $\pm \lambda$. Of course, for Misiurewicz parameters we have that the Julia set is the whole plane, so both critical values belong to the Julia set.

A  natural and interesting question in this setting, therefore, is  to study the connectivity of the escaping set for Misiurewicz parameters of the exponential family. In this paper we show that, unlike  Misiurewicz parameters in the sine (or cosine) family, the escaping set is connected. 

\begin{thmA}
Let $\lambda$ be a Misiurewicz parameter. Then
$\mi(E_{\la})$ is a connected subset of the plane.
\end{thmA}

\noindent I would like to thank R. L. Devaney and L. Rempe for their helpful comments. I am also grateful to  the hospitality of the  Mathematics Department at Boston University for their hospitality during the preparation of this paper. Finally I  also wish to thank the Spanish {\it Ministerio de Ciencia y Innovacion} for the financial support.

\section{Preliminaries}\label{section:preliminaries}

Before proving  the connectedness of the escaping set for Misiurewicz parameters of the exponential family we need to establish  some terminology and state some previous results that we will  use in the next section.

We say that a point $z_{0}\in  \mj\left(E_{\lambda}\right)$ is {\it accessible} if there exist a curve $\gamma:[0,\infty) \to \mj\left(E_{\lambda}\right)$ with $\gamma(0)=z_{0}$ and $\gamma(t)\in \mi\left(E_{\lambda}\right)$ for all $t>0$, and Re$\left (\gamma(t)\right)\to \infty$ as $t\to \infty$. In this context we say that such a curve $\gamma$ connects $z_0$ to $+\infty$.

\begin{Remark}
For points in the Julia set there is another standard notion of accessibility when such a point can be reached through a curve inside the Fatou set. Here we use accessibility in another sense, first introduced, as far as we know, in \cite{R2}.  
\end{Remark}

We emphasize that the key argument when proving the existence of a curve $\gamma$ in the definition of accessibility at a point in $\bc$, is to show that certain curves contained in the escaping set, called {\it tails}, defined far to the right in the dynamical plane and with certain associated certain symbolic dynamics can be extended by using a pull back argument (i.e. taking suitable preimages of the tails by the inverse branches of the exponential map). These extended curves are called {\it hairs} or {\it dynamical rays}. In some cases it is possible to show that the hairs {\it land} at a unique point  of the Julia set, so called {\it endpoint (of the hair}) which is the only point in $\gamma$ that does not need to be in the escaping set.  Notice that when the hair lands at a certain point, this point becomes accessible in the sense of the previous definition. Again a detailed proof of this construction can be found in \cite{B,R,SZ,SZ1}. 

The following result, which precisely reflects the above discussion, is a particular case  of Theorem 2.3 and Theorem 4.3 (or Theorem 6.4) in \cite{SZ1}. (See also Theorem 6.5 in \cite{SZ}).  Remember that $\la$ is a Misiurewicz parameter if the orbit of $0$ is strictly preperiodic and that in this case $\mj\left(E_{\lambda}\right) = \bc$.

\begin{Proposition}\label{prop:0-accessible}
Fix a  Misiurewicz parameter $\lambda\in \bc$. Then there is one preperiodic hair (or dynamic ray) landing at $0$, and so  the singular value is accessible (by a curve $\gamma$).  In particular $\gamma$  extends to $+\infty$  with an asymptotic constant imaginary part.  Finally, since $\gamma$  lands at $0$, there is a (initial) piece of $\gamma$ that belongs to an arbitrarily small neighborhood of $0$.  
\end{Proposition}

This result was  proven in \cite{B} for a certain subset of Misiurewicz parameters. Also,  in both papers \cite{B,SZ1} (and references therein) the existence of  the  curve $\gamma$ connecting the singular value to infinity is proven for many other parameter values, and, for those parameters, for a large subset of points in the Julia set (see Proposition \ref{prop:z-accessible} below). 

\begin{figure}[ht]
	\centering
	\includegraphics[width=150pt]{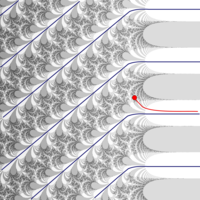}
	\setlength{\unitlength}{150pt}
	\put(-0.38,0.54){\small $z=0$}
	\put(-0.15,0,48){\small $\gamma$}
	\put(-0.95,0,07){\footnotesize $E_{\la}^{-1}\left(\gamma\right)$}
	\put(-0.95,0,34){\footnotesize $E_{\la}^{-1}\left(\gamma\right)$}
	\put(-0.95,0,61){\footnotesize $E_{\la}^{-1}\left(\gamma\right)$}
	\put(-0.95,0,85){\footnotesize $E_{\la}^{-1}\left(\gamma\right)$}
	\caption{\small This is the dynamical plane for $\la=2\pi i$, the {\it easiest} Misiurewicz point: $0  \mapsto 2\pi i \mapsto 2\pi i$.  Different graduations from white to black means different speeds in escaping to infinity. We (over)draw the curve $\gamma$ and some of the components of $E_{\la}^{-1}\left(\gamma\right)$.}
\end{figure}
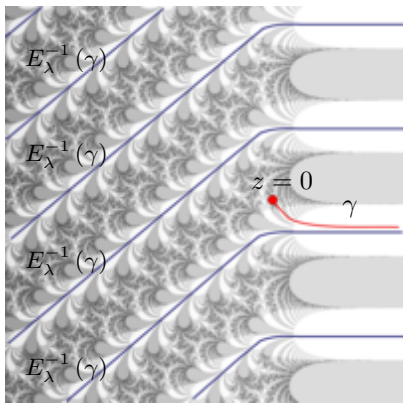

We use the hair $\gamma$ connecting $0$ with $+\infty$ to divide the plane into $2\pi i$ periodic fundamental domains where we will define an appropriate symbolic dynamics. Since $0$ is an omitted value (it has no finite preimages and its only {\it moral} preimage is $-\infty$) the set $E_{\lambda}^{-1}\left(\gamma\right)$ is the union of infinitely many disjoint curves that extend to infinity in both directions. From Proposition \ref{prop:0-accessible} we know that each component of $E_{\lambda}^{-1}\left(\gamma\right)$ extends to $+\infty$ with an asymptotically  constant imaginary part (in fact each component of $E_{\lambda}^{-1}\left(\gamma\right)$ is asymptotic to a straight line Im$(z)=2\pi k$ for some $k\in \bz$), while, as we will see, it extends to $-\infty$ with certain bounded slope.  Finally the set $\bc \setminus E_{\lambda}^{-1}\left(\gamma\right)$ consists of infinitely many open strips labeled by $R_{k}, \ k\in\bz$, so that $R_{k}$ contains $2\pi i k$, and  $E_{\lambda}: R_{k} \mapsto \bc \setminus \gamma$ is a conformal isomorphism for all $k$.

From this division of the dynamical plane it is natural to define the set all infinite sequences $\Sigma = \{r=r_0r_1\ldots \},\ r_i\in \bz$.  For all  $z\in\bc$ whose forward orbit does not meet the boundaries of the $R_{k}$'s, we denote by $r(z)=r_{0}r_{1}r_{2}\ldots , \ r_{j}\in\bz$ its {\it itinerary (with respect to this partition)}  so that $r_{j}=n$ if and only if $E_{\lambda}^j(z)\in R_{n}$. Consequently, all points in $\bc$ have a well defined itinerary unless the point is eventually  mapped to the boundary of some $R_{k}$. In this case, however,  the point belongs to the escaping set and its dynamics is well understood. In particular the itinerary of $0$, and all points in $\gamma$, is always well defined and strictly preperiodic (remember that from Proposition \ref{prop:0-accessible} we know that $\gamma$ is a preperiodic dynamic ray landing at $0$). 

We say that $r$ is {\it exponentially bounded} if there exists $\hat{x}>0$ such that $2\pi |r_{j}|<E_{|\lambda|}^j\left(\hat{x}\right)$ for all $j\geq 0$.  For instance all periodic or preperiodic sequences are exponentially bounded. It is well--known that only exponentially bounded  sequences in $\Sigma$ correspond to realizable itineraries with respect to the {\it static partition} (i.e., the straight lines corresponding to the  preimages of the negative real part); see \cite{B,SZ,V}.

As we have mentioned, far to the right, the static partition of the plane and the one given by the curves $E_{\lambda}^{-1}\left(\gamma\right)$ asymptotically differs from a constant, so points whose orbit tends to infinity must have exponentially bounded itineraries. In contrast, for an arbitrary accessible singular value, we have no a priori control from the construction on the way the curves $E_{\lambda}^{-1}\left(\gamma\right)$ extend to $-\infty$ and so it is not immediate that {\it all} points in $\bc$ have exponentially bounded itineraries. In could be the case that the slope of the curves $E_{\lambda}^{-1}\left(\gamma\right)$ grows more than exponentially (increasing or decreasing) when approaching $-\infty$. However, if $\lambda$ is preperiodic, this approach to $-\infty$ has a bounded slope and this guarantees that all points in $\bc$ have an exponentially bounded itinerary.  

The next lemma collects the key aspects of the above discussion.

\begin{Lemma} \label{lemma:exp_bounded}
Let $\lambda\in\bc$ be a Misiurewicz parameter.
\begin{enumerate}
\item[(a)] If  $z\in\bc$ has itinerary $r(z)$, then $r(z)$ is exponentially bounded. 
\item[(b)] If $z\in\bc$ is a periodic (respectively preperiodic) point and  $w\in \bc \setminus \mi\left(E_{\la}\right)$ has the same periodic (respectively preperiodic) itinerary as $z$, then $w=z$.
\end{enumerate}
 \end{Lemma}

\begin{proof}
First we prove (a). Fix $\la\in \bc$ Misiurewicz and let $0,z_0,\ldots z_{m-1},\overline{z_m,\ldots z_{m+p}}$, $m\geq 1,\ p\geq 0$, be the preperiodic orbit of $0$, where $\overline{z_m,\ldots z_{m+p}}$ means the string $z_m,\ldots z_{m+p}$ repeated infinitely many times. Let  $\mu=\left(E_{\lambda}^{p+1}\right)^{\prime}\left(z_m\right)$ and  take the linearizing coordinates around the point $z_m$ that conformally conjugates, via a conformal map $\Phi$, $E_{\lambda}^{p+1}$ to the linear map $z\to \mu z$, in a small neighborhood of $z_m$. Since $\gamma$ lands at $0$, its corresponding image by $E_\la^{m}$, $\gamma_m$, must land at $z_m$. Using the (local) linearizing coordinates $\gamma_m$ spirals around $z_m$ satisfying $\mu\Phi\left(\gamma_m\right)= \Phi\left(E_\la^{p+1}\left(\gamma_m\right)\right)$, and so it spirals with bounded speed governed by $\mu$. Consequently $\gamma$  (a suitable preimage of $\gamma_m$ landing at $0$) spirals around $0$ at  an exponentially bounded speed (in a sufficiently small neighborhood of $0$). Assuming it spiral clockwise (if it spirals counterclockwise the arguments are similar), the curves  $E_{\lambda}^{-1}\left(\gamma\right)$ tend to $-\infty$ with asymptotic exponentially  bounded negative slope, i.e. there exist constants $c>0$, $d>1$ and $m<0$ such that for all $z\in R_k, \ k\in\bz$, with Re$(z)<m$ it holds that 
\begin{equation}\label{eq:exp_bounded}
c\ {\rm Re}(z)-d+2k\pi \leq {\rm Im}(z) \leq  c\ {\rm Re}(z)+d+2k\pi.
\end{equation}
If $r=r_0r_1 \ldots r_n \ldots$ was not exponentially bounded, for any $x>0$ there would be  a (sub)sequence  of positive symbols $r_{j_k}\to\infty$  such that $2\pi r_{j_k} /E_{|\la|}^{r_{j_k}}(x)> M_{j_k}$ with $M_{j_k}\to \infty$ as $j_k\to\infty$. Let $z=x+iy, \ x>0$. It is easy to see that $|{\rm Im} \left(E_{\la}^{j}(z)\right)|\leq |E_{\la}^{j}(z)| \leq E_{|\la|}^j(x)$ for all $j>0$. On the other hand, for  $j_k$ large enough, if $w\in R_{r_{j_k}}$ with $|{\rm Re}(w)|\leq E_{|\la|}^{j_k}(x)$ we have from (\ref{eq:exp_bounded}) that 
$$
{\rm Im}(w) \geq c \ {\rm Re}(w)-d+2\pi r_{j_k} \geq -c \ E_{|\la|}^{j_k}(x) - d +M_{j_k}E_{|\la|}^{j_k}(x) > E_{|\la|}^{j_k}(x).
$$
Thus there are no points following a non exponentially bounded sequence (neither escaping infinity nor with oscillating orbit).  
    
Now we will prove (b). Because the branch of the inverse map taking values in $R_k$, i.e.  $E_{\lambda,k}^{-1}: R_{k} \mapsto \bc \setminus \gamma$, is a conformal isomorphism for all $k$, it is enough to show the statement when 
$z$ is a periodic point. Let $r:=r(z)=\overline{r_0\ldots r_{n-1}}$  be the (periodic) itinerary of $z$, and let $w$ be a non escaping point with the same itinerary, i.e. $r(z)=r(w)$.

Let 
$$
 \mathcal T_k^{m,M} =\{ z\in R_k \ | \ m\leq {\rm Re}(z) \leq M \} \quad {\rm and} \quad \mathcal T^{m,M}:=\bigcup_{k\in\{r_0,\ldots r_{n-1}\}} \mathcal T_k^{m,M}$$

We claim that there exists  $m <0$ small enough and $M>1$ large enough such that  all points of the orbit of $z$ and $w$ must lie in $\mathcal T^{m,M}$.  This claim follows since $z$ and $w$ are non--escaping points with periodic itinerary of period $n$, while the itinerary of all points to the left of (a small enough)  $m<0$ must coincide with the (strictly preperiodic) itinerary  of $0$ for an arbitrary number of symbols and so it cannot be periodic of period $n$.  
{\it A priori} $\mathcal T^{m,M}$ could have more than one connected component (since arbitrarily  far to the left the components of $E_\la^{-1}\left(\gamma\right)$ could have folds). However,  since an initial portion of $\gamma$ belongs to an arbitrarily small neighborhood of $0$ (see Proposition \ref{prop:0-accessible}), we may choose $m<0$ such that  only one of the connected components of $\mathcal T^{m,M}$  extends to Re$(z)=M$, and such that there are no points belonging to any other connected component of $\mathcal T^{m,M}$ that have a periodic itinerary of period $n$. We again denote by $\mathcal T^{m,M}$ the connected component with this property. 
Finally we choose, if necessary, a larger $M>0$ so that the right hand side boundary of $\mathcal T^{m,M}$ maps to the right of itself. To simplify notation we use $ \mathcal T_k:= \mathcal T_k^{m,M}$ and $\mathcal T := \mathcal T^{m,M}$. Finally we denote by $E_{\la,\mathcal T_k}^{-1}$ the brach of the inverse of $E_\la$ taking values in $\mathcal T_k$.

For all $k\in\{r_0,\ldots r_{n-1}\}$, let $\varphi_k:\mathbb D \to \mathcal T_k$ be the Riemann mapping, where $\mathbb D$ is the open unit disc. Consider the sequence of maps
$$
\Phi_{\ell} := \varphi_{r_0}^{-1} \circ \left( E_{\la,\mathcal T_{r_0}}^{-1}\circ \ldots \circ E_{\la,\mathcal T_{r_{n-1}}}^{-1}  \right)^{\ell} \circ \varphi_{r_0}, \ \ell >0,
$$ 
where $\left( E_{\la,\mathcal T_{r_0}}^{-1}\circ \ldots \circ E_{\la,\mathcal T_{r_{n-1}}}^{-1}  \right)^{\ell}$ denotes the composition of the maps repeated $\ell$ times. Clearly, $\Phi_{\ell},\ \ell>0$, are well--defined  maps from the unit disc to itself. Notice that $z$ (and $w$) must belong to $T_{r_0}$ and have itinerary $r$. Moreover, as $\ell$ tends to infinity, points in the image of $\Phi_\ell \left(\mathbb D\right)$ must follow the periodic sequence $r$ for an arbitrarily large number of symbols. 

We claim that, for sufficiently large $\ell$, the maps $\Phi_\ell$'s are strict contractions in the Poincar\'e metric of $\mathbb D$, and so $w=z$. 
The claim follows from two considerations. On the one hand, the right--hand--side boundary of each of the $\mathcal T_k$ maps outside $\mathcal T$. On the other hand, points in a preassigned (small) neighborhood  of  the upper, lower and left--hand--side boundaries of each one of the $\mathcal T_k$'s must agree  with the  strictly preperiodic sequence of $0$ for an arbitrarily large number of symbols while, for $\ell$ sufficiently large, the points in the image of $\Phi_\ell$ agree with (the periodic sequence) $r$ for an arbitrarily large number of symbols. 
\end{proof}

Statement (b) in the previous lemma also follows from Lemma 3.2 in 
\cite{R3S}, and from Proposition 4.4 in \cite{SZ1} when $w$ is assumed to also be periodic. 

We finish this section by stating (only for Misiurewicz parameters) the following fact, which is proven in \cite{SZ} (precisely, Corollary 6.9), concerning the structure of the escaping set. 

\begin{Proposition}\label{prop:z-accessible}
Let $\lambda\in\bc$ be a Misiurewicz parameter for $E_{\lambda}$. If $z \in  \mi\left(E_{\lambda}\right)$ then $z$ is accessible (i.e. it must belong to a certain hair or it is the endpoint of a hair).
\end{Proposition}

\section{Proof of Theorem A}

Let $\lambda$ be a Misiurewicz parameter and  $\gamma$ the curve that makes $0$ accessible (in fact from Proposition \ref{prop:0-accessible} we know that $\gamma$ is a preperiodic hair). In what follows when considering boundaries and closures they will be relative to the complex plane (rather than the extended complex plane).

The proof of the main result will be done by contradiction. If  $\mi(E_\la)$ was not a connected subset of the plane, there would exist an open connected set $U\subset \bc$ such that the following three conditions would be satisfied 
\begin{equation}\label{eq:definition-U}
\begin{split}
& {\rm (a)} \ \ \mi(E_\la) \cap U  \ne \emptyset \\
& {\rm (b)} \ \ \mi(E_\la) \cap \partial U  = \emptyset \\
& {\rm (c)} \ \  \mi(E_\la) \not\subset  U 
\end{split}
\end{equation}

\begin{Lemma} \label{lemma:unbounded}
If a connected set $U$ satisfies (\ref{eq:definition-U}), then $U$ and $\bc\setminus U$ must be unbounded sets.
\end{Lemma}

\begin{proof}
Because $U$ satisfies (\ref{eq:definition-U}) there are points of the escaping set in $U$ and $\bc\setminus U$. From Proposition \ref{prop:z-accessible}  any point  in the escaping set is accessible.  So if $U$ or $\bc\setminus U$ were bounded,  the curve connecting any point of the escaping set with infinity would cross the boundary of $U$, a contradiction to condition 1(b). 
\end{proof}

\begin{Lemma}\label{lemma:new-V}
If there exists an (unbounded) open connected set $U$ satisfying (\ref{eq:definition-U}), then there exists an unbounded open connected set $V\subset \bc$ satisfying (\ref{eq:definition-U})  and having an unbounded connected boundary. Moreover $U_{1}=\bc \setminus \overline{V}$ is an unbounded open connected subset of the plane such that  $\mi(E_\la)\cap U_{1}\ne \emptyset$ and $\partial U_{1}=\partial V$.
\end{Lemma}

\begin{proof}
Start with any set $U$ satisfying (\ref{eq:definition-U}). From Lemma \ref{lemma:unbounded} we know it is unbounded. Let us consider $\bc \setminus \overline{U}$. This is a union of open connected sets, denoted by $V_{\omega},\ \omega\in\Omega$, each of which has a connected boundary (otherwise $U$ would be disconnected). From Lemma \ref{lemma:unbounded}  we conclude that at least one of them, $V$, must be unbounded and such that $V \cap \mi(E_\la) \ne \emptyset$, since otherwise conditions (\ref{eq:definition-U}b) and (\ref{eq:definition-U}c) would not be satisfied simultaneously for $U$. Since $\bc$ is unicoherent we have that $\partial V$ is an unbounded connected set (see \cite{W}, Chapter 4). Finally we define $U_{1}:=\bc \setminus \overline{V}$.  
\end{proof}

From the above we may assume that $V$ is  an unbounded, open set satisfying (\ref{eq:definition-U}), $U_{1}=\bc \setminus \overline{V}$ is an  unbounded, open connected subset of the plane such that  $\mi(E_\la)\cap U_{1}\ne \emptyset$, and its common boundary $\mc:=\partial U_{1}=\partial V$ is an unbounded, closed, connected set.

\begin{Lemma}\label{lemma:no-cuts}
For each $k\geq 0$,  $E_\la^k(\mc)\cap \partial R_{j} = \emptyset$ for all $j\in\bz$. In other words for each $k\geq 0$ we have that $E_\la^k(\mc) \subset R_{j}$ for some $j\in\bz$. 
\end{Lemma}

\begin{proof}
The proof is straightforward since all points in $\partial R_{j},\ j \in \bz$, belong to the escaping set while all points in $E_\la^k(\mc), \ k\geq 0$ do not.
\end{proof}

\begin{Lemma}\label{negative-real-part}
Let  $r(0)=s_0s_1\ldots$  
be the preperiodic itinerary of $0$. Then, for all $n\in\bn$, there exists $\mu$ such that if Re$(z)<\mu$ then $r(z)=r_{0}r_{1}\ldots$ satisfies $r_{j}=s_{j-1}$ for all $j=1,\ldots n$.
\end{Lemma}

\begin{proof}
Fix any $n\geq 1$. The exponential map $E_\la$ sends any vertical line Re$(z)=\rho$ to a circle of radius $|\la| {\rm exp}(\rho)$. So, choosing $\rho_{0}<<0$ sufficiently small each half plane Re$(z)<\rho_{0}$ is mapped into a small disk around $z=0$. Hence 
$r_{1}=s(0)=0$. So by continuity of  $E_\lambda$,  and choosing, if necessary, $\mu\leq \rho_{0}<0$,  we guarantee that for all  points in Re$(z)<\mu$ we have  $r_{j}=s_{j-1}$ for  $j=2,\ldots n$.  
\end{proof}

\begin{Lemma}\label{lemma:left}
Fix $k\geq 0$. Then, if $z\in E_\la^k(\mc)$ there exist $\rho=\rho(k)$ such that Re$(z)>\rho$. In other words $E_\la^k\left(\mc\right), \ k\geq 0$ cannot be unbounded to the left. 
\end{Lemma}

\begin{proof}  Since $0$ is preperiodic we can write its preperiodic itinerary as  $r(0)=s_0\ldots s_{m-1}\overline{s_m\ldots s_{m+p}}$, with  $m\geq 1$, $p\geq 0$. Of course $s_0=0$.

Let $k\geq 0$. Let us fix $w\in E_\la^k\left(\mc\right)$ and consider its itinerary $r(w)=t_{0}t_{1}t_2\ldots$. Since $w$ is not an escaping point we have from  Lemma  \ref{lemma:exp_bounded}(b) that there exists $n\geq 1$ such that $t_j=s_{j-1}$ for all $1\leq j < n$ but $t_{n} \ne s_{n-1}$. 

If $E_\la^k\left(\mc\right)$ were unbounded to the left then, from Lemma \ref{negative-real-part}, we know that for all $\ell>0$ there exists a point $z\in E_\la^k\left(\mc\right)$ such that $r(z)=r_0r_1\ldots$ where $r_j=s_{j-1}$ for all $j=1\ldots \ell$. Consequently if $z\in E_\la^k\left(\mc\right)$ is chosen so that $\ell=n$ then $r_{n}\ne t_{n}$.  This would imply that $E_{\lambda}^{n}(w)$ and $E_{\lambda}^{n}(z)$ (both in $E_\lambda^{k+n}\left(\mc\right)$) belong to different $R_j$-strips, a contradiction with Lemma \ref{lemma:no-cuts} and the fact that $E_\lambda$ maps connected sets to connected sets.
\end{proof} 

\noindent Now we can prove the connectedness of the escaping set.

\begin{Theorem}\label{theorem:main}
Let $\lambda\in\bc$ be a Misiurewicz parameter. Then $\mi(E_\la)$ is a connected subset of the plane.
\end{Theorem}

\begin{proof}
From Lemma \ref{lemma:left}  we conclude that for each $k\geq 0$, $E_\la^k\left(\mc\right)$ cannot be  unbounded to the left. Since we also know that $\mc$ is unbounded and belongs to a certain $R_j, \ j\in\bz$, it must be unbounded to the right. The same applies to all its images $E_\la^k\left(\mc\right)$ (since each of them belong to a certain $R_{j(k)}$). 

We will show that under these hypotheses there must a point in $\mc$ whose orbit escapes, which is a contradiction. We know that all points in $\mc$ follow the same exponentially bounded itinerary, say $r=r_{0}r_{1}\dots$. Therefore,  there exists $\hat{x}>0$ such that $2\pi |r_{j}|<E_{|\la|}^j(\hat{x}),\ j\geq 0$.

From Proposition \ref{prop:0-accessible} we know that far enough to the right the components of  $E_{\lambda}^{-1}\left(\gamma\right)$ have an asymptotic constant imaginary part. Let $\rho\geq \hat{x}$ be such that   $\mc \cap \left({\rm Re}(z)=x\right)\neq\emptyset$. For all $j\geq 0$, let $B_{r_j}$ be the closed connected region in  $R_{r_{j}}$ that is bounded above and below by $\partial R_{r_{j}}$, and left  and right by Re$(z)=E_{|\la|}^j\left(\rho\right)$ and 
Re$(z)=E_{|\la|}^j\left(\rho\right)+2\pi$, respectively. Clearly  $E_\la\left(B_{r_{j}} \right)$ is an annulus and we claim that  $B_{r_{j+1}}\subset E_\la\left(B_{r_{j}} \right)$. For the case of the static partition of the plane, this claim is proven in \cite{V} (Lemma 2.4). A similar argument, which we include here for completeness, also works in our setting.  Each $B_{r_{j}}$ lies in the sector $|{\rm Im}(z)|\leq {\rm Re}(z)$ (this follows immediately from the inequality $2\pi |r_{j}|<E_{|\la|}^j\left(\rho\right),\ j\geq 0$) and that the outer circle of the annulus meets the lines $y=\pm x$ at points with real part $\left(\sqrt{2}/2\right) e^{2\pi}E_{|\la|}^{j+1}\left(\rho\right)$ which turns out to be larger than $E_{|\la|}^{j+1}\left(\rho\right)+2\pi$. Moreover $|E'(z)|>1$ for all $z$ with Re$(z)>\rho$. 

If $z_{1}$ is a point in $E_\la\left(\mc \right)\cap B_{r_{1}}$ then, by construction, there must be a point $z_{0}^1$ in $\mc \cap B_{r_{0}}$ such that $E_\la^{-1}\left(z_{1} \right)=z_{0}^1$. Similarly, if $z_{2}$ is a point in $E_\la^2\left(\mc \right)\cap B_{r_{2}}$ then  there must be a point $z_{1}^{\prime}$ in $E_\la\left(\mc \right)\cap B_{r_{1}}$ such that $E_\la^{-1}\left(z_{2} \right)=z_{1}^\prime$. So, there must be a point  $z_{0}^2$ in $\mc \cap B_{r_{0}}$ such that $E_\la^{-2}\left(z_{2} \right)=z_{0}^2$. By applying the same construction we can  construct a sequence of points $z_{0}^k$ in $\mc \cap B_{r_{0}}$ such that $E_\la^{k}\left(z_{0}^k \right)=z_{k}$. Each of these points has an orbit whose first $k$ iterates move to the right in the asymptotic direction. If we take 
$$
z_{0}=\lim\limits_{k\to \infty} z_{0}^k 
$$
then $z_{0}\in B_{r_{0}}\cap \mc$ (because $B_{r_{0}}$ is compact and $\mc$ is closed), and the orbit of $z_{0}$ escapes, since Re$\left(E_\la^k\left(z_{0}\right)\right) \to \infty$, a contradiction.
\end{proof}

During the refereeing process, Lasse Rempe,  using some of the ideas of a preprint  version of this paper, has extended the connectivity of the escaping set to other  parameter values of the exponential family.

\end{document}